\documentclass[12pt]{amsart}
\usepackage{amsthm}
\usepackage{amsmath}
\usepackage{amsfonts}
\usepackage{amssymb}
\usepackage{color}

\newcommand{\beq}{\begin{equation}}
\newcommand{\beqn}{\begin{equation*}}

\newcommand{\eeq}{\end{equation}}
\newcommand{\eeqn}{\end{equation*}}

\newcommand{\R}{\mathbb{R}}
\newcommand{\C}{\mathbb{C}}
\newcommand{\Q}{\mathbb{Q}}

\newcommand{\rk}{{\rm rk}}

\newcommand{\sone}{S^1}
\newcommand{\z}{\mathbb{Z}}
\newcommand{\Z}{\mathbb{Z}}

\newtheorem{theorem}{Theorem}
\newtheorem{lemma}{Lemma}
\newtheorem{corollary}{Corollary}

\newtheorem{proposition}{Proposition}
\newtheorem{remark}{Remark}
\title[The second closed geodesic]{The
	second closed geodesic,
		the fundamental group, and generic
		Finsler metrics}
\author{Hans-Bert~Rademacher}
\address{Mathematisches Institut, Universit{\"a}t Leipzig,
	04081 Leipzig, Germany}
\email{rademacher@math.uni-leipzig.de}
\author{Iskander~A.~Taimanov}
\thanks{The second author (I.A.T.) was supported by Russian Science Foundation (grant 19-11-00044).}
\address{Novosibirsk State University,
	630090 Novosibirsk, Russia, and Sobolev Institute of Mathematics,
	630090 Novosibirsk}
\email{taimanov@math.nsc.ru}
\date{2020-11-09}

\begin{document}

\begin{abstract}
For compact manifolds with infinite
fundamental group we present sufficient
topological or metric conditions ensuring
the existence of two geometrically
distinct closed geodesics.
We also show how results about generic
Riemannian metrics can be carried over to
Finsler metrics.

\end{abstract}
\subjclass[2010]{53C22, 58E10}
\keywords{closed geodesic, fundamental group, generic metric, Finsler metric}

\maketitle
\section{Introduction}
In this paper we are interested in existence results
for closed geodesics of a Riemannian or Finsler metric
on a compact manifold with infinite fundamental group.
The shortest closed curve in a non-trivial
free homotopy class
is a closed geodesic. Hence on a manifold with non-trivial
fundamental group there always exists a
non-contractible closed geodesic.
It is not difficult to see that
there are infinitely many geometrically distinct
and non-contractible
closed geodesics if the first Betti number
satisfies
$b_1(M)=\rk H_1(M;\z)\ge 2.$
The existence of infinitely many geometrically distinct
closed geodesics in case of manifolds with
fundamental group
$\pi_1(M)\cong \z$ was shown by Bangert and Hingston,
cf.~\cite{BH}, and for an infinite solvable fundamental group
in~\cite{T1993}.
In \cite{RT} the authors present results about the existence of
infinitely many
non-contractible and
geometrically distinct closed
geodesics on connected sums $M=M_1\# M_2.$

We investigate which topological assumptions
imply the existence of at least two geometrically
distinct and non-contractible
closed geodesics for compact manifolds with
infinite fundamental group.
\begin{theorem}
	\label{thm:riem}
Let $M$ be a compact manifold with infinite fundamental
group $\pi_1(M)$ without torsion elements
whose first homology group
$H_1(M;\z)$ is neither trivial (then the fundamental
group is perfect) nor cyclic of finite order
(i.e. $H_1(M;\z)\cong \z/(m\z)$ for some $m \ge 2$).
Then any Riemannian or reversible Finsler metric
carries at least two geometrically distinct
and non-contractible
closed geodesics. 	
\end{theorem}
If the fundamental group does not satisfy the assumptions of
Lemma~\ref{lem:group-anosov}, i.e.
if the group is not isomorphic
to $\z$ and
if there is no element
$a$ of infinite order such that any other element is
conjugate to a power of $a$
then it is clear that
there are two geometrically distinct closed geodesics,
cf. \cite[(3.15)]{Ba}.
It is not known whether a
finitely presented group with this property
exists, cf.
Remark~\ref{rem:kourovka}.
In Theorem~\ref{thm:riem} we give
sufficient
conditions for the first homology group of the manifold
ensuring the existence
of a second non-contractible closed geodesic.

We also present an existence result for
non-reversible Finsler metrics. For Riemannian metrics
the \emph{bumpy metrics theorem} due to Abraham~\cite{Ab}
resp. Anosov~\cite{An} and perturbation arguments due
to Klingenberg and Takens~\cite{KT} can be used to prove
that for a $C^4$-generic metric either all closed geodesics
are \emph{hyperbolic} or there is a non-hyperbolic closed
geodesic of \emph{twist type}. In the second case the
Birkhoff-Lewis fixed point theorem implies the existence
of infinitely many geometrically distinct closed geodesics,
cf. \cite[Thm.3.3.10]{Kl} and \cite{Mo}.

For Finsler metrics corresponding genericity statements
hold as we explain in Section~\ref{sec:generic}.
For $r\ge 4$ a $C^r$-generic Finsler metric on a compact
manifold is bumpy, cf. Theorem~\ref{thm:bumpy}.
And for $r\ge 6$ a $C^r$-generic Finsler metric on
a compact manifold either all closed geodesics are hyperbolic, or there exists a non-hyperbolic
closed geodesic of twist type, cf.
Theorem~\ref{thm:twist}.
Hence for a statement about the existence of a second closed
geodesic at least for a generic non-reversible Finsler
metric it is sufficient to consider metrics all of whose
closed geodesics are hyperbolic.
\begin{theorem}
	\label{thm:nonrev}
	Let $M$ be a compact manifold with non-trivial fundamental
	group $\pi_1(M)$
	endowed with a non-reversible
	Finsler metric.
	If at least one of two following conditions:
	
	a) the first Betti number of $M$ is positive: $b_1(M) = \dim H_1(M;\R) > 0$;
	
	b) all closed geodesics of this metric are hyperbolic,
	
	\noindent
	holds then there are two geometrically distinct
	and non-contractible closed geodesics.
\end{theorem}
For simply-connected and compact manifolds with a bumpy
and non-reversible Finsler metric
there are at least two geometrically distinct
closed geodesics. This was shown by
Duan, Long and Wang in~\cite[Cor.1.2]{DLW}.
It also follows from existence results for
closed orbits of non-degenerate Reeb flows
obtained by Abreu, Gutt, Kang \&
Macarini~\cite[Cor.1.14]{AGKM}.

The structure of the paper is as follows:
After explaining the setting in Section~\ref{sec:setting}
and discussing generic properties of Finsler metrics
in Section~\ref{sec:generic} we prove
Theorem~\ref{thm:riem} in Section~\ref{sec:riem}.
Finally we prove Theorem~\ref{thm:nonrev} in
the last section.
\section{Preliminary facts}
\label{sec:setting}

\subsection{Functional spaces}
Let $M^n$ be a
compact Riemannian or a Finsler manifold.

We denote by
$\Lambda(M^n)=H^1(S^1,M)$ the space of $H^1$-maps
$$
\gamma: [0,1] \to M^n, \ \ \ \gamma (0)=\gamma(1),
$$
of a circle $S^1 = \R/\Z$ into $M^n$ and by
$\Omega_x (M^n)$ the subspace of
$\Lambda(M^n)$ formed by loops starting and ending at $\gamma(0)=\gamma(1) =x \in M^n$.

Let $\Pi^+(M^n)$ and $\Pi(M^n)$ be the quotients of $\Lambda(M^n)$ with respect to the $SO(2) (=S^1)$-action:
$$
\varphi \cdot \gamma(t) = \gamma(t+\varphi), \ \ \ \varphi \in S^1 = \R/\Z,
$$
and the $O(2)$ action generated by $SO(2)$ and the inversion of the parameter:
$$
\gamma^{-1}(t) = \gamma(-t).
$$

Let $g: \Lambda M \to \Lambda M$
{map a curve to }{
be the map which corresponds to a curve} the same line
parametrized proportionally to the arc-length, such that $( g\cdot \gamma)(0) = \gamma(0)$.
We put
$$
L(M) = g(\Lambda(M)),
$$
and denote by $P^+(M)$ and $P(M^n)$ the following quotient-spaces:
$$
P^+(M) = L(M)/SO(2), \ \ \ \ \  P(M)= L(M)/O(2).
$$
It is known that these spaces are deformation retracts of $\Lambda(M)$, $\Pi^+(M)$ and $\Pi(M)$, respectively.
We recall that the action of $SO(2)$ is not free due to the iterates of
prime curves.

There is the procedure which replaces $L(M)/SO(2)$ by the quotient of a free action of $SO(2)$.
We recall it in the full generality. Let $X$ be a topological space
on which the group $G$ acts continuously.
The homotopy quotient $X_G$ of the $G$-space $X$ is the quotient of the product
$X \times EG$ with respect to the diagonal action of $G$. Here
$$
EG \stackrel{G}{\longrightarrow} BG
$$
is the universal $G$-bundle. By definition, $EG$ is contractible.
The $G$-equivariant cohomology are defined as
$$
H^\ast_G(X) = H^\ast(X_G).
$$
For $G=SO(2)$ we have $BG = \C P^\infty$, and, therefore, the $SO(2)$-equivariant cohomology
of the point with the trivial $SO(2)$ action on it are isomorphic to the cohomology of $\C P^\infty$:
$$
H^\ast_{SO(2)}(\mathrm{pt};\Z) = H^\ast(\C P^n;\Z) = \Z[u], \ \ \deg u =2.
$$
For $G=SO(2)$ with the standard action by left multiplications we have
$$
H^\ast_{SO(2)}(SO(2);\Z) = H^\ast(\mathrm{pt};\Z) = H^0(\mathrm{pt};\Z) = \Z,
$$ compare for example~\cite{Hi} or \cite{Ra89}.

\subsection{The exact homotopy sequence for the spaces of non-contractible closed curves}

Given $h \in \pi_1(M,x_0)$, we denote by
$[h] \in [S^1,M]$ the corresponding free homotopy class of closed  curves and by
$$
\Lambda M[h] \subset \Lambda M \ \ \ \mbox{and} \ \ \ LM[h] \subset LM
$$
the connected components of $\Lambda M$ and $LM$ consisting of curves from $[h]$.

Let
$h$ be realized by a map $\omega: [0,1] \to M$ with $\omega(0) = x_0$, and
let $h_i$ be the automorphism
$$
h_i: \pi_i(M,x_0) \to \pi_i(M,x_0)
$$
corresponding to the standard action of $h \in \pi_1$ on $\pi_i$.

We have

\begin{theorem}[\cite{Ballmann},\cite{T1985}]
The mapping
$$
\pi: \Lambda M \to M, \ \ \ \pi(\gamma) = \gamma(0),
$$
which maps a closed curve $\gamma$ onto the marked point  $\gamma(0)$,
is a Serre fibration with the fibre $\Omega M$:
$$
\Lambda M \stackrel{\Omega M}{\longrightarrow} M.
$$
The exact homotopy sequence for this fibration restricted onto  $\Lambda M[h]$
takes the form
\begin{equation}
\label{seq}
\dots \to \pi_i (\Lambda M [h], \omega) \stackrel{\pi_\ast}{\longrightarrow} \pi_i (M,x_0)
\stackrel{f_i}{\longrightarrow} \pi_{i-1} (\Omega_{x_0}(M), \omega) =
\end{equation}
$$
= \pi_i(M,x_0)
\to \pi_{i-1} (\Lambda M[h], \omega) \to \dots
$$
where

a) $\pi_\ast(\pi_i(\Lambda M[h], \omega)) = \mathrm{St}\,(h_i)$, where $\mathrm{St}\,(h_i)$ is the subgroup of
$\pi_i(M, x_0)$ consisting of all elements fixed under $h_i$;

b) $f_i = h_i - \mathrm{id}$ for $i \geq 2$.

Moreover, the homotopy sequence finishes as
$$
\to \pi_2(M) = \pi_1(\Omega M[h]) \to \pi_1(\Lambda M[h]) \to \mathrm{St}(h) \to 1,
$$
where $\mathrm{St}(h)$ is the subgroup of $\pi_1(M)$ formed by all elements which commute with $h$.
\end{theorem}
The maps $
f_i$ are written uniformly in the simple form
$$
f_{i}(g) = [h,g], \ \ \ g \in  \pi_{i}(M,x_0), \ \ \ i \geq 1,
$$
where
$[h,g]$ is the Whitehead product of $h \in \pi_1$ and $g \in \pi_{i}$.

By the Serre theorem,
$$
\pi_{i-1}(\Omega M) = \pi_i(M),
$$
and for the simplest case $h=1$ (the unit of $\pi_1$) we have $h_\ast = \mathrm{id}$ and the exact sequence splits into short exact subsequences
of the form
$$
0 \to \pi_{i+1}(M) \to \pi_i(\Lambda M) \stackrel{\pi_\ast}{\longrightarrow} \pi_i(M) \to 0.
$$

In the sequel we take tensor products of all entries with $\Q$ over $\Z$. Therewith all higher homotopy groups become
$\Q[\pi_1]$-modules (a priori such modules can be even infinitely generated) which we denote by $\pi^\Q_i$.

\subsection{The type numbers of closed geodesics} A point $\gamma \in \Lambda M$
is a closed geodesic if it is a critical point of the length functional
$$
S(\gamma) = \int _\gamma F(x,\dot{x})dt,
$$
where $F(x,v), x \in M, v \in T_xM$, is positive for $v \neq 0$ and homogeneous of  first order in $v$:
$$
F(x,\lambda v) = \lambda F(x,v) \ \ \mathrm{for} \ \ \lambda \geq 0,
$$
and the set $S_xM=\{v\in T_xM\,;\,F(x,v)=1\}\}
	\subset T_xM$
	of unit vectors is a strictly convex subset.
This action functional is called the Finsler length.
If
$$
F(x,v) = \sqrt{g_{ik}(x)v^iv^k},
$$
where $g_{ik}$ is the Riemannian metric tensor, we get the usual length with respect to the Riemannian metric.
The Finsler metric is called
\emph{reversible} if $F(x,v) = F(x,-v)$ for all $x$ and $v \in T_xM$.

The classical Morse theory is applicable to the study of critical points of $S$ on the spaces $\Lambda M$.
We omit the exposition of the standard things and only discuss an important statement due to Schwarz.

It is clear that every iterate of a closed geodesic is a closed geodesic itself.
For a reversible Finsler metric the inverse $\gamma^{-1}$ of the geodesic $\gamma$ is also geodesic.
For a non-reversible Finsler metric we say that two closed geodesics are
\emph{(geometrically) distinct} if they are not iterates of the same
closed curve with the same
orientation. Geodesics of a Riemannian metric or, more general, reversible Finsler metric are (geometrically) distinct if they
are not iterates of a certain geodesic and its inverse.

A metric is called \emph{bumpy} if all its closed geodesics are non-degenerate in the Morse
theoretic sense except the trivial one-point extremals.

We assign to every nontrivial closed geodesic its \emph{type numbers} which are
$$
t_i = \mathrm{rank}\, H_i\left(\{S < S_0\} \cup \gamma,\{S<S_0\};\Q\right) \ \ \ \mbox{where $S_0=S(\gamma)$}
$$
and $S$ is considered as a functional on $\Lambda/SO(2)$ or on $\Lambda_{SO(2)}$ (on the quotient or on the homotopy quotient of $\Lambda = \Lambda M$).
It is evident that for nondegenerate closed geodesics $t_i \in \{0,1\}$.

It was established by Schwarz that, if $\gamma$ is a prime closed geodesic, then
$$
t_i (\gamma^k) = \delta_{i,\mathrm{ind}\,\gamma^k}
$$
if the Morse index of $\gamma^k$, i.e., $\mathrm{ind}\,\gamma^k$, has the same parity as $\mathrm{ind}\, \gamma$
and
$$
t_i (\gamma^k) = 0 \ \ \ \mathrm{otherwise}.
$$
Therefore, we have

\begin{proposition}
\label{prop-schwarz} \cite{Schwarz}
If for two nondegenerate closed geodesics $\gamma_1$ and $\gamma_2$, we have
$$
t_i(\gamma_1) =1, \ \ \  t_j(\gamma_2) = 1, \ \ \ i-j \equiv 1~\mod\, 2,
$$
then these geodesics are distinct.
\end{proposition}

This proposition is very useful in combination with another statement derived from the Bott index formula for iterates of geodesics:

\begin{proposition}
\label{prop-bott} \cite{Bott56}
If $\gamma$ is a nondegenerate
prime closed geodesic, then the indices of $\gamma$ and its iterates of odd order have the same parity:
$$
\mathrm{ind}\, \gamma^{2k+1} - \mathrm{ind}\,\gamma \equiv 0 \mod 2.
$$
\end{proposition}

\subsection{$\pi_1(\Lambda M[h])$ and geometrically distinct closed geodesics in $\Lambda M[h]$}

\begin{proposition}
\label{oldsmj} \cite{T1993}
If
$\mathrm{St} (h)$
is non-cyclic, then there are at least two
geometrically distinct closed geodesics in $\Lambda M[h])$.
\end{proposition}

\begin{proof}
First, let us recall the original argument from \cite{T1993} (see Lemma 3).
Let two loops $\gamma$ and $\eta$ realize elements $h$ and $g$ from $\mathrm{St}$ which does not lie
in a cyclic subgroup of $\mathrm{St}(h)$. Moreover we assume that $\gamma$ realizes a closed geodesic of minimal length in this homotopy class.
Take a mapping of a two-torus
$T^2 \to M$ such that one of the parallels is mapped into $\gamma$ and another parallel is mapped into
$\eta$. This torus is swept out by closed curves homotopic to $\gamma$ and
deforming this loop in $\Lambda[h] (=\Lambda M[h])$ we have either to stop at some nontrivial closed geodesic $\gamma_1$ or to deform the whole loop into $\gamma$.  In the latter case the second parallel is deformed
to some iterate of a curve $\gamma_0$ such that $\gamma$ is also an iterate of $\gamma_0$.
However that contradicts the choice of $\eta$.
Hence we have a closed geodesic in $\Lambda [h]$ which is distinct from $\gamma$.

Algebraically this argument is reformulated as follows.
Let us consider the homotopy quotient $\Lambda[h]_{SO(2)}$.
We have the fibration
\begin{equation}
\label{fibration}
\Lambda[h] \times ESO(2) \stackrel{S^1}{\longrightarrow} \Lambda[h]_{SO(2)}
\end{equation}
and from the exact homotopy sequence for this fibration we have
\begin{equation}
\label{caseb}
\to \pi_1(S^1) \to \pi_1(\Lambda[h]) \to \pi_1(\Lambda[h]_{SO(2)}) \to 1.
\end{equation}
By construction, the image of $\pi_1(S^1)$ is just the cyclic subgroup generated by $h$.
Hence
\begin{equation}
\label{caseb2}
\pi_1(\Lambda[h]_{SO(2)}) \neq 1
\end{equation}
and therefore $\Lambda[h]_{SO(2)}$ is not contractible and contains critical points different from the minimal closed geodesic.
\end{proof}

\begin{corollary}
\label{coroldsmj}
If the metric is bumpy and the conditions of Proposition \ref{oldsmj} hold then there are two
geometrically distinct closed geodesics in $\Lambda M[h]$.
\end{corollary}

\begin{proof}
Indeed, from (\ref{caseb2}) and the bumpyness of the metric it follows that there is closed geodesic of index one. By Propositions \ref{prop-schwarz} and \ref{prop-bott},  it has to be geometrically distinct from the minimal
closed geodesic in this homotopy class.
\end{proof}
\section{
	Generic conditions for Riemannian and Finsler metrics and the existence of closed geodesics}
\label{sec:generic}
In this section we discuss how known generic
properties of Riemannian metrics are also
generic for Finsler metrics.
A Finsler metric $F: TM\longrightarrow \R$
defines for a point $x \in M$
a family of Riemannian metrics $g^y$ on the tangent space
$T_xM$ parametrized by unit vectors
$y \in S_xM:=\{w \in T_xM\,;\, F(w)=1\}.$
For a coordinate system $x=(x_1,\ldots,x_n)$
for a chart $U\subset M$ of the manifold $M$
with corresponding chart $(x,y)=(x_1,\ldots,x_n,
y_1,\ldots,y_n)$ for the tangent bundle $TU$
the metric coefficients $g_{ij}(x,y)=g_{ij}^y(x)$
are given by:
\begin{equation}
\label{eq:finsler-metric-coeff}
g_{ij}(x,y)=g^y_{ij}(x)=
\frac{1}{2}
\frac{\partial F^2}{\partial y_i \partial y_j}(x,y)\,.
\end{equation}
For a closed geodesic $c:\R \longrightarrow M$
of the Finsler metric
parametrized by arc length
we can extend the
velocity vector field $c'$ along $c$ to a non-zero
vector field $V$ in a tubular neighborhood
$U\subset M$ of
the closed geodesic.
Then the Riemannian metric $g^V$ is called
\emph{osculating Riemannian metric.}
Now we follow the notation of Anosov~\cite{An}.
We choose Fermi coordinates $(u_0,u)=
(u_0,u_1,\ldots,u_n)$ (here $n+1=\dim M$)
in a tubular neighborhood
of the closed geodesic with respect to the
osculating Riemannian metric $g^V.$
Hence $u_0(t)=t, u(t)=0$
parametrizes the closed geodesic
and for any $t_0$
and any $v \in \R^n$ the curve
$u_0(t)=t_0, u(t)=tv$ is a geodesic, too.
Then for the corresponding metric coefficients
we obtain
\begin{equation}
g_{ij}^V(t,0)=\delta_{ij}, 0\le i,j\le n\,;\,
\frac{\partial }{\partial u_k}g_{ij}^V(t,0)=0\,.
\end{equation}
Then $c$ is also a closed geodesic of $g^V$ with the same
length. The
\emph{flag curvature} $K(V(x);\sigma)$ of the Finsler
metric of a flag $(V(x),\sigma)$ with a plane $\sigma
\in T_xM, V(x)\in \sigma$ agrees with the sectional curvature
$K(\sigma)$
of the osculating Riemannian metric $g^V,$
cf.~\cite{Ra04}.
Parallel transport along $c$ as well as the index
and nullity of the closed geodesic of the Finsler metrics
agree with
the parallel transport and
the index and nullity of the osculating Riemannian
metric.
This allows to extend the proof of the bumpy metric theorem
for Riemannian metrics presented by Anosov~\cite[Thm.1]{An}
to the Finsler case.
This seems to be obvious to experts, for
example Ginzburg and G\"urgel mention that the
statement follows from Anosov's proof without
giving details, cf.~\cite[Example 1.15]{GG}.
In the sequel we indicate how this extension works.

For a compact manifold $M$ denote by
$\mathcal{F}^r(M)$ the set of $C^r$-Finsler metrics
$F: TM \longrightarrow \R$
endowed with the strong $C^r$-topology.
Similar to the \emph{bumpy metrics theorem}
for Riemannian manifolds due to Abraham~\cite{Ab} resp.
Anosov~\cite{An} the corresponding
statement holds for Finsler metrics:
\begin{theorem}
	\label{thm:bumpy}
For a compact differentiable manifold a $C^r$-generic Finsler for $r\ge 4$
metric is bumpy.
\end{theorem}
In contrast to the Riemannian case the order $r$ is at least
$4$ instead of $2.$ This is due to the fact that
Equation~\eqref{eq:finsler-metric-coeff}
shows that we obtain the osculating metric
by differentiating the Finsler metric twice
fibrewise.
We can extend the proof
given by Anosov using the following argument:
Let $c$ be a closed geodesic
$c$ of the Finsler metric $F$ with
an osculating Riemannian metric $g^V$
in a tubular neighborhood of $c$
defined for the vector field $V$
as above and let
 $x \in U
\mapsto \delta g(x)$
be a field of symmetric bilinear forms
on the tangent space $T_{x}M$
with $\delta g_{00}(u_0,0)=2\sum_{i=1}^n
f_i(u_0,u)u_i$ and
$\delta g_{ij}(u_0,u)=0$ for all
$i,j; i+j\ge 1.$
Here $f_i, i=1,\ldots,n+1$ are smooth functions with support
in a neigborhood of a point $(t_0,0),$
cf. \cite[\S 5]{An}.
We define
a one-parameter family
$s\in (-\epsilon,\epsilon)
\mapsto F_s$ of Finsler metrics
with $F=F_0$ by the
Equation:
\begin{equation}
F^2_s(u_0,u,y_0,y)=
F^2(u_0,u,y_0,y)+
s \phi(y) \sum_{k=1}^nf_k(u_0,u)u_k
y_0^2\,.
\end{equation}
Here
$\phi: \R^n \longrightarrow [0,1]$ is
$0$-homogeneous
smooth
function which is locally constant
with value $1$ in a neighborhood
of $(1,0,\ldots,0).$
Hence $F_s$ is a
smooth
perturbation of the Finsler metric
$F$ with the property that
$c$ is a closed geodesic for all Finsler metrics
$F_s.$ The osculating Riemannian metric
$g_s^V$ of $F_s$ with respect to the vector field
$V$ near $c$
is given by:
\begin{equation}
(g_s^V)_{ij}(u_0,u)=
g^V_{ij}(u_0,u)+\delta_{i0}\delta_{j0}
s\sum_{k=1}^n f_k(u_0,u)u_k\,.
\end{equation}
Then for all $s$ the curve $c$ is still a geodesic
for $F_s$
and the osculating Riemannian metric
$g_s^V$ near the closed geodesic $c$
is of the form: $g_s^V=g^V+s \delta g.$
Therefore Lemma 1 and Lemma 2 in Anosov's
proof of the bumpy metrics theorem carry over
to the Finsler setting which is sufficient
for the proof of
the Theorem.
In ~\cite{So} it was already shown that one can
perturb a closed geodesic of a Finsler metric
to become non-degenerate using the
osculating Riemannian metric.
The proof of
Theorem~\ref{thm:bumpy}
could also follow the approach presented by
Rifford and Ruggiero~\cite{RR}.
The \emph{linearized Poincar\'e map} $P_c$
of a closed geodesic is a linear symplectic map.
If $c$ is not hyperbolic denote by
$\exp(2\pi i\lambda_j(c)),
\lambda_j(c) \in [0,1/2], 1\le j\le l, l\le n-1$ the eigenvalues of
$P_c$ of modulus $1.$
The Finsler metric is called
\emph{strongly bumpy} if all eigenvalues
of the prime closed geodesics are simple and
if any finite set of the disjoint union of
the Poincar\'e exponents $\lambda_j(c)$
is algebraically independent,
cf.~\cite[Def.4.1]{Ra94}.
Then the proofs in~\cite{Ra94} carry over
to Finsler metrics and starting from
Theorem~\ref{thm:bumpy} one obtains:
\begin{theorem}
Let $M$ be a compact manifold.

\smallskip

(a) A $C^4$-generic Finsler metric is
strongly bumpy, cf.~\cite[Thm.4.3]{Ra94}

\smallskip

(b) A strongly bumpy metric on a simply-connected manifold carries infinitely many geometrically
distinct closed geodesics,
cf.~\cite[Thm.5.7]{Ra94}
\end{theorem}
Using the \emph{Birkhoff normal form} of a
symplectic mapping one defines the concept of a
non-hyperbolic closed geodesic of
(generalized) \emph{twist type.}
It is an open
and dense condition
for the $3$-jet of the Poincar\'e mapping
of a non-hyperbolic closed geodesic.
For a $C^4$-Riemannian metric on a compact manifold
either all closed geodesics are hyperbolic or
there exists a closed geodesic of
\emph{twist type.}
This can be shown starting from a bumpy metric
by perturbing the
$3$-jet of the Poincar\'e mapping of
single closed geodesics,
cf. \cite{KT} or \cite[Thm. 3.3.10]{Kl}.
Similar arguments can be used to show
the following
\begin{theorem}
\label{thm:twist}
Let $M$ be a compact manifold.

\smallskip

(a)
For $r\ge 6$ a  $C^r$-generic Finsler metric
either all closed
geodesics are hyperbolic or there exists
a non-hyperbolic closed geodesic of
twist type.

\smallskip

(b) In a tubular neighborhood of a non-hyperbolic closed geodesic of twist type there a
infinitely many geometrically distinct closed
geodesics.
geode
\end{theorem}
The necessary local pertubation argument
in (a)
in the larger class of \emph{Tonelli Hamiltonians}
was shown by Carballo and Miranda~\cite[Cor.5]{CM}.
The order $6$ instead of $4$ as in the
Riemannian case comes from taking the
Hamiltonian $H: T^*M \longrightarrow \R$ corresponding to
the Lagrangian $L=F^2:TM \longrightarrow
\R$ by applying the Legendre transformation:
$y \in T_xM \longrightarrow g^y(y,.)\in
T^*M,$ cf.
Equation~\eqref{eq:finsler-metric-coeff}.
Part (b) is a direct consequence of the
\emph{Birkhoff-Lewis fixed point theorem,}
cf.~\cite[Thm.3.3.10]{Kl} and \cite{Mo}.

In~\cite[p.141]{Zi} Ziller sketches how Pugh's 
\emph{closing lemma} \cite{Pu}
implies that for a
$C^2$-generic Finsler metric on a compact manifold
the initial vectors to closed geodesics are dense
in the unit tangent bundle,
cf. also~\cite{PR} for a 
Hamiltonian closing lemma.

Hence for these metrics there are in particular
infinitely many geometrically distinct closed geodesics.
The perturbations used in the closing lemma
do not preserve particular
flows as the geodesic flow of a Riemannian metric
or a magnetic geodesic flow. A $C^1$-closing lemma for
Riemannian metrics was obtained
by Rifford~\cite{Ri}.
\section{The second closed geodesic of a Riemannian or reversible Finsler metric}
\label{sec:riem}

By using Propositions \ref{prop-schwarz} and \ref{prop-bott} and computations of certain homologies in \cite{Schwarz}, Fet proved that a bumpy Riemannian metric
on a simply-connected closed manifold has at least two distinct closed geodesics \cite{Fet}.
For compact manifolds with infinite fundamental group
the analogue of this statement is not yet established.

In this and the next sections we present topological and generic conditions under which the existence of
second closed geodesic on a non-simply-connected manifold can be proved.

We want to give a sufficient condition for the fundamental
group of a compact manifold to ensure the existence of a
second closed geodesic.
On a compact manifold with a non-trivial fundamental group
$\pi_1(M)$
and with a torsion element $a \in \pi_1(M)$ (i.e. $a^m=1$ for some
$m \ge 2$) there are infinitely many closed geodesics for
a Riemannian metric all of
	whose closed geodesics are hyperbolic. This is shown in
~\cite[Thm.A]{BTZ1981}, cf.
Proposition~\ref{pro:btz}.
If we assume that a compact
manifold with non-trivial fundamental group
and without torsion elements
is endowed with an arbitrary Riemannian metric
which carries only one
closed geodesic then its fundamental group
satisfies the assumption of the following algebraic
lemma, cf. \cite[(3.15)]{Ba}:
\begin{lemma}
\label{lem:group-anosov}
Let $G$ be a group with the following property:
There is an element $a \in G$ such that
the subgroup
$\langle a\rangle=\{a^m; m\in \z\}\subset G$
generated by $a$ is isomorphic to the group of
integers, i.e. $\langle a \rangle
\cong \z$ and such that
for all $g \in G$ there is an element $h \in G$ and
$m \in \z$ with $g=h a^m h^{-1}.$

Then the abelianization $
G^{ab}=H_1(G;\z)=G/[G,G]$ is finite and cyclic
unless $G\cong \z.$ Hence the commutator subgroup $[G,G]$
is infinite and has
finite index.
\end{lemma}
\begin{proof}
If $G \not\cong \z$ then there is $\widetilde{b} \in G$ and
$h \in G, m\in \z$ with $\widetilde{b}=h a^m h^{-1} \not\in \langle a \rangle .$
It follows that $b:=hah^{-1} \not\in \langle a
\rangle.$
Then
$$ 1 \not= b a^{-1}=h a h^{-1} a^{-1}=[h,a].$$
On the other hand there is an element $f \in G, k \in \z, k\not=0$ with
$$ 1 \not= b a^{-1}=[h,a]=f a^k f^{-1}.$$
Then
$$ [h,a]a^{-k}=[f,a^k]$$
which implies
$$ a^k =[a^k,f][h,a] \in [G,G].$$
Hence the normal subgroup generated by $a^k$ is a subgroup
of the commutator subgroup $[G,G]$ and
the quotient group $H_1(G;\z)=G/[G,G]$ is a finite
cyclic group whose order divides $k.$
\end{proof}
\begin{proof}[Proof of Theorem~\ref{thm:riem}
	in the Introduction]
Assume there is a prime closed geodesic
$c$ such that all closed geodesics are geometrically
equivalent to $c.$
Then there is a nontrivial homotopy class
$a \in \pi_1(M)$ such that $c$ lies in the free homotopy
class generated by $a$ and such that the following
holds:
For every $g \in \pi_1(M)$ there exists $h\in \pi_1(M),
m \in \z$ with $g=h a^m h^{-1}.$
By assumption the subgroup
$\langle a \rangle$ is isomorphic to $\z.$
If $\pi_1(M)\cong \z$
then there are infinitly many geometrically distinct closed geodesics by a result due to
Bangert-Hingston,
cf.~\cite{BH}.
Hence the fundamental group satisfies the assumptions of
Lemma~\ref{lem:group-anosov}. Therefore the abelianization
$H_1(M;\z)$ of the fundamental group is either trivial
(then the fundamental group is perfect) or it is
isomorphic to $\z_m, m\ge 2.$
\end{proof}
\begin{remark}
	\label{rem:kourovka}	
\rm
Problem 8.8 in the Kourovka notebook~\cite{Kou}
posed by Anosov asks whether a
finitely generated group satisfying
the assumption of Lemma~\ref{lem:group-anosov}
besides the integers
exists. Guba~\cite{Gu} constructed such a group, but this group is not
finitely presented and hence it is not the fundamental group
of a compact manifold,
cf. \cite[p.206]{T1992}
and \cite[Rem.(3) following Thm.B.]{BTZ1981}.

We notice if such a group exists and appears as the fundamental group of a closed $K(\pi,1)$ manifold,
then
we do not see how to prove the existence
of at least two closed geodesics on such a manifold
with the methods used in this text.
\end{remark}

\section{The second closed geodesic
of a non-reversible Finsler metrics}
\label{sec:nonrev}

In this section we consider \emph{non-reversible
Finsler metrics.} In particular two closed geodesics
$c_1,c_2:\sone\longrightarrow M$ are {\em geometrically
equivalent} if their images coincide, i.e.
$c_1(\sone )=c_2(\sone)$ and if their orientations coincide.

In case of non-reversible Finsler metrics a result by
Ballmann, Thorbergsson and Ziller gives the following
\begin{proposition}
\cite[Thm.A]{BTZ1981}
\label{pro:btz}
Let $M$ be a compact manifold with
non-trivial fundamental group $\pi_1(M).$
Let $a \in \pi_1(M), a \not=1.$  We assume that for some {\em positive}
integers $0<n<m$ the conjugacy classes  $[a^n]=[a^m]$
of the iterates $a^n,a^m$ coincide.
If the manifold carries a non-reversible
Finsler metric all of whose closed geodesics are hyperbolic then
\begin{equation}
	\label{eq:prime-growth}
\liminf_{t\to \infty} \frac{N^a(t)}{t}>0.
\end{equation}
\end{proposition}
Here $N^a(t)$ is the number of geometrically
distinct closed geodesics of length $\le t$ which
are freely homotopic to a power $a^k, k\ge 1.$
The proof of the Birkhoff-Lewis fixed point
theorem~\cite{Mo} also implies that for
a closed geodesic $c$ of twist type freely
homotopic to $a \in \pi_1(M)$ there are
infinitely many closed geodesics
freely homotopic to a power
$a^k$ and
the function $N^a(t)$ satisfies
Equation~\eqref{eq:prime-growth}.
Using this Proposition we give the following
\begin{proof}[Proof of Theorem~\ref{thm:nonrev} in the Introduction]
Since the metric is
non-re\-ver\-sible, we say that two closed geodesics are geometrically distinct if they are
not positive iterates $c^k$ and $c^l$, $k,l>0$, of a closed curve $c$.

If condition a) holds then the minimal closed curves in the homology classes $[a]$ and $[a^{-1}]$,
where $[a]$ is of infinite order in $H_1(M)$, are definitely geometrically distinct.

We are left to consider the case when
condition b) holds.
We assume that there is only one prime closed geodesic
$c:\sone \longrightarrow M$.
Choose $h \in \pi_1(M)$ such that $c$ is freely homotopic to $h$: $c \in \Lambda M [h]$.
If $h$ is a torsion element in $\pi_1$, then Proposition \ref{pro:btz} implies that there are infinitely many
geometrically distinct closed geodesics. Hence we assume that
\begin{equation}
\label{eq:torsion}
h^m\not=1 \ \ \ \mbox{for all $m=1,2,\dots$}.
\end{equation}
Since there is only one closed geodesic there
exists a $k \ge 1$ such that the closed geodesic
$c^k$ is up to parametrization the unique shortest
closed curve in the free homotopy class $[h^{-1}]$.
Hence
there exists
$a \in \pi_1(M)$ such that for some $k \ge 1:$
\begin{equation}
\label{eq:aminus}
h^{-1}=a h^k a^{-1}.
\end{equation}
Then we conclude
\begin{eqnarray*}
h&=&(h^{-1})^{-1}= a h^{-k} a^{-1}= a^2 h^{(k^2)} a^{-2}.
\end{eqnarray*}
If $k \ge 2$ then we obtain from Proposition~\ref{pro:btz}
that there are infinitely many closed geodesics.

Hence we can assume $k=1.$
There are two commuting elements $h,a^2 \in \pi_1(M)$ and we
can assume that both are of infinite order.
If they do not lie in some cyclic subgroup, then by Proposition \ref{oldsmj}, there
exist two geometrically
distinct closed geodesics.
Therefore we are left to assume that $h$ and $a^2$ belong to some cyclic subgroup.
We notice that $h$ is a prime element, otherwise $h = g^m$ for some $g \in \pi_1$ and $m >1$
and the homotopy class $[h]$ contains an $m$-th iterate of the minimal closed geodesic $c^\prime$  in $[g]$.
In this case either there are at least two closed geodesics $c$ and $c^\prime$, or
$c^\prime = c^l, l>0$, and $c$ is conjugate to $c^{lm}$ which implies, by Proposition~\ref{pro:btz},
the existence of infinitely many closed geodesics.

Therefore $a^2 =h^q$ for some $q>0$. By  (\ref{eq:aminus}), for $k=1$ we have
$$
h^{-q} = a h^q a^{-1} = a a^2 a^{-1} = a^2 = h^q,
$$
which implies that
$$
h^{2q}=1,
$$
which contradicts Equation~\eqref{eq:torsion}.
\end{proof}

\end{document}